\documentclass{amsart}
\usepackage{amssymb}
\usepackage{graphicx}
\usepackage{amscd}

\theoremstyle{plain}

\newtheorem{corollary}{Corollary}
\newcommand{\N}{\mathbb{N}}

\newtheorem{lemma}{Lemma}

\newtheorem{proposition}{Proposition}
\newtheorem{remark}{Remark}

\newtheorem{theorem}{Theorem}
\newtheorem{defn}[theorem]{Definition}

\newcommand{\To}{\longrightarrow}
\def\Ext{\operatorname{Ext}}
\newcommand{\adef}{\begin{defn}}
\newcommand{\zdef}{\end{defn}}

\title[Extension and lifting of operators and polynomials]{Extension and lifting of operators\\ and polynomials}

\author{Jes\'us M.F. Castillo}
\address{Departamento de Matem\'aticas, Facultad de Ciencias, Univ. de Extremadura, Avda Elvas s/n 06011, Badajoz, Espa\~{n}a}
\email{castillo@unex.es}

\author{Ricardo Garc\'ia}
\address{Departamento de Matem\'aticas, Facultad de Ciencias, Univ. de Extremadura, Avda Elvas s/n 06011, Badajoz, Espa\~{n}a}
\email{rgarcia@unex.es}

\author{Jes\'us Su\'arez}
\address{Departamento de Matem\'aticas, Escuela Polit\'ecnica de C\'aceres, Univ. de Extremadura, Avda Universidad s/n 07001, C\'aceres, Espa\~{n}a}
\email{jesus@unex.es}

\thanks{This research has been supported in part by Project MTM2010-20190-C02-01 and Junta de Extremadura GR10113  ``IV Plan Regional I+D+i, Ayudas a Grupos de Investigaci\'on'' .}
\begin{document}
\maketitle

\begin{abstract}
We study the problem of extension and lifting of operators
belonging to certain operator ideals, as well as that of their
associated polynomials and holomorphic functions. Our results
provide a characterization of $\mathcal{L}_1$ and
$\mathcal{L}_{\infty}$-spaces that includes and extends those of
Lindenstrauss-Rosenthal \cite{LR} using compact operators and
Gonz\'{a}lez-Guti\'{e}rrez \cite{GG} using compact polynomials. We
display several examples to show the difference between extending
and lifting compact (resp. weakly compact, unconditionally
convergent, separable and Rosenthal) operators to operators of the
same type. Finally, we show the previous results in a homological
perspective, which helps the interested reader to understand the
motivations and nature of the results presented.
\end{abstract}

\section{Introduction}

Many authors have considered the problem of extension and lifting
of operators \cite{dom,Fa,Ka,K,L, LR,LT, St}, of homogeneous polynomials
\cite{AB, CGV, CiG, D, GG} and holomorphic mappings \cite{A, AB,
CGV, D,Ka, Z} in Banach spaces. We present here a unifying method
of proof for most of those results, and several new ones, using
some tools of homological algebra.

To this end, recall that a short exact sequence of Banach spaces
and linear continuous operators is a diagram $$
\begin{CD}
0@>>>Y@>i>>X@>q>>Z@>>>0
 \end{CD}$$
where the image of each arrow coincides with the kernel of the
following one. The open mapping theorem ensures that $Y$ must then
be a subspace of $X$ ($i$ is an injection) and $Z$ is the
corresponding quotient $X/Y$  ($q$ is a quotient map). The exact
sequence is said to \emph{split} if $Y$ is complemented in $X$;
which means that there is a linear continuous projection $p: X\to
Y$. Let $\mathfrak L$ denote the class of all linear continuous
operators. The projection $p$ is, by definition, a linear
continuous extension of the identity of $Y$ and thus ``$Y$ is
complemented in $X$" is equivalent to ``for every Banach space $E$
every linear continuous operator $t\in \mathfrak L(Y, E)$ can be
extended to a linear continuous operator $T\in \mathfrak L(X, E)$
through $i$; i.e., $Ti=t$". It is part of the folklore --see
\cite[1.1]{CG}-- that this is equivalent to ``for every Banach
space $V$ every linear continuous operator $t\in \mathfrak L(V,
Z)$ can be lifted to a linear continuous operator $T\in \mathfrak
L(V, X)$ through $q$; i.e., $qT=t$".

We are interested in considering the situation when we replace the
class $\mathfrak L$ by another class $\mathcal A$ of operators,
polynomials or holomorphic mappings acting between Banach spaces.\\

\adef We will say that an exact sequence $0 \to Y \stackrel{i}\to
X \stackrel{q}\to Z \to 0 $ $\mathcal{A}$-splits if, for every
Banach space $E$, every $t \in \mathcal A(Y, E)$ can be extended
to a $T \in \mathcal{A}(Y,E)$ through $i$; i.e., $T i = t$. We
will say that it $\mathcal{A}$-lifts if, for every Banach space
$V$, every $t \in \mathcal A(V,Z)$ can be lifted to a $T \in
\mathcal{A}(V, X)$ through $q$; i.e., $qT = t$.\zdef

We will say that an exact sequence $0 \to Y \stackrel{i}\to X
\stackrel{q}\to Z \to 0 $ \emph{uniformly $\mathcal{A}$-splits}
when there exists $\lambda>0$ such that, for every Banach space
$V$, every $t \in \mathcal A(Y, V)$ can be extended to a $T \in
\mathcal{A}(Y,V)$ through $i$ with $\|T\|\leq \lambda \|t\|$.
Analogously, we will say that it \emph{uniformly
$\mathcal{A}$-lifts} if there exists $\lambda>0$ such that, for
every Banach space $V$, every $t \in \mathcal A(V,Z)$ can be
lifted to a $T \in \mathcal{A}(V, X)$ through $q$ with
$\|T\|\leq \lambda \|t\|$.\\

Observe that when $\mathcal A$ is a closed ideal of $\mathfrak L$
then $\mathcal A$-splitting and uniform $\mathcal A$-splitting
coincide; as well as $\mathcal A$-lifting and uniform $\mathcal
A$-lifting. This obviously fails for non-closed ideals: for
instance, when $\mathfrak F$ is the ideal of finite rank
operators, all exact sequences $\mathfrak F$-split while, as it
will be clear after
definition 2, not all exact sequences uniformly $\mathfrak F$-split.\\

In this paper we study when an exact sequence $\mathcal{A}$-splits
or $\mathcal{A}$-lifts for the following choices of $\mathcal{A}$:
the ideal $\mathfrak{F}$ of finite rank operators, $\texttt{A}$ of
approximable operators, $\mathfrak K$ of compact operators,
$\mathfrak W$ of weakly compact operators and $\mathfrak L$ of all
linear continuous operators; for their associated polynomial
ideals $\mathcal{P}_{\mathfrak K}$ of compact polynomials and
$\mathcal{P}_{\mathfrak W}$ of weakly compact polynomials; and for
their associated classes of holomorphic bounded mappings
$\mathcal{H}_{\mathfrak K}$ and
$\mathcal{H}_{\mathfrak W}$ (see sections 3 and 4).\\

The paper is organized as follows: Section 3 contains the main
results about extension and lifting of operators. Given an exact
sequence $0 \to Y \stackrel{i}\to X \stackrel{q}\to Z \to 0$ of
Banach spaces, the Hahn-Banach theorem  guarantees  that the
restriction operator $i^*:X^*\to Y^*$ ($i^*(x^*)=x^*|_{Y}$) is
surjective. We show that the existence of a linear continuous
operator $s:Y^*\to X^*$ such that $i^*s=Id_{Y^*}$ is equivalent to
any of the following conditions (see Theorem \ref{locallysplit}):
the sequence locally splits (see Definition 2), uniformly
${\mathfrak F}$-splits, $\texttt{A}$-splits, ${\mathfrak
K}$-splits, $ {\mathfrak W}$-splits, uniformly ${\mathfrak
F}$-litfs or $\texttt{A}$-lifts. If moreover $Z$ has the Bounded
Approximation Property (in short, BAP), the previous conditions
will be shown to be also equivalent to $\mathfrak K$-lifts. This
result unifies and extends Kaballo \cite[Thm. 3.4]{Ka} and
Fakhoury \cite[Thm. 3.1]{Fa}. It is then proved that $\mathfrak
K$-splitting and $\mathfrak K$-lifting re not equivalent. The BAP
has a decisive role in the lifting of operators. In Proposition
\ref{bap} it is proved: \textit{A separable Banach space $Z$ has
the BAP  if and only if for every exact
sequence $0\rightarrow Y \rightarrow X \rightarrow Z \to 0$ $\mathfrak K$-spliting and ${\mathfrak{K}}$-litfing are equivalent}.\\

Equipped with these results and the representation of polynomials
by symmetric tensor products, section 4 is devoted to prove
extension/lifting results for polynomials and holomorphic bounded
mappings. The injective tensor product
$\check{\otimes}_{\varepsilon }$ (resp. projective $\hat{\otimes
}_{\pi  }$) has deep connections with vector functions spaces
since it often occurs that injective (resp. projective) tensor
product spaces can be represented as vector function spaces (see
\cite{B,Ka,DF,F}). We will prove that the notions locally
splitting, $\mathcal{P}_{\texttt{A}}$-splitting,
$\mathcal{P}_{\mathfrak K}$-splitting, $\mathcal{P}_{\mathfrak
W}$-splitting, $\mathcal{H}_{\mathfrak K}$-splitting,
$\mathcal{H}_{\mathfrak W}$-splitting and
$\mathcal{P}_{\texttt{A}}$-lifting are all equivalent. If moreover
$Z$ has the BAP, the previous conditions are also equivalent to
$\mathcal{P}_{\mathfrak{K}}$-lifting and
$\mathcal{H}_{\mathfrak{K}}$-lifting (see Theorem \ref{ls}). In
Section 5 we  provide  characterizations of $\mathcal{L}_1$ and
$\mathcal{L}_{\infty}$-spaces that include and extend those of
Lindenstrauss-Rosenthal \cite[Thm. 4.1]{LR} using compact
operators and Gonz\'{a}lez-Guti\'{e}rrez \cite[Thms.  2  and 4]{GG} using compact
polynomials and Domanski \cite[Thm. 4]{dom}.\\

Section 6 contains the most interesting counterexamples in the
paper. Its motivation is to make explicit the difference between
``extending operators belonging to a certain operator ideal to
operators of the same operator ideal" and the same property for
liftings. This problem responds to a very natural homological
problem as we explain in Section 7. The connection between the
extension/lifting problem for operator ideals has appeared, often
not explicitly, in many papers (see e.g., \cite{LR,dom}). We will
call an operator ideal $\mathfrak A$ \emph{balanced} when
extension and lifting are equivalent for operators in $\mathfrak
A$. For technical reasons explained in Section 6, this notion is
only meaningful for injective and surjective (see below for the
definition) operator ideals. We will then study the balanced
character of the main classes of injective and surjective operator
ideals appearing in the literature: compact, weakly compact,
unconditionally convergent, separable and Rosenthal operators. The
ideal of compact operator is the only one balanced.

Section 7 puts previous results in a homological perspective,
which helps the interested reader to understand the motivations
and nature of the results presented.

\section{Preliminaries}

For general information about operator ideals we suggest \cite{Pi}
or, more friendly, \cite{djt}. Recall that an operator ideal
$\mathcal A$ is a subclass of the class $\mathfrak L$ such that
for all Banach spaces $V$ and $X$ its components $\mathcal A(V,X )
= \mathfrak L(V, X ) \cap \mathcal A$ satisfy: $\mathcal A(V,X )$
is a linear subspace of $\mathfrak L(V, X )$ which contains the
finite rank operators and enjoys the ideal property: for $u\in
 \mathfrak L(V, X ), t\in \mathcal A(X,W), w\in
 \mathfrak L(W, Y )$, the composition $wtu \in
\mathcal A(V,Y)$. An operator ideal  $\mathcal A$ is injective
(resp. surjective) whenever given an operator $t\in  \mathfrak
L(V, Y )$ and an injection $i:Y\to X$ (resp. surjection $q:X \to
Z$) then $t\in \mathcal A(V,Y)$ if and only if $it \in \mathcal
A(V,X)$ (resp. $t\in \mathcal A(Z,V)$ if and only if  $tq\in
\mathcal A(X,V)$)).

All necessary background information and unexplained notation
about polynomials and holomorphic mappings can be found in
\cite{D}; more specific information about homogeneous polynomial
ideals can be seen in \cite{Pi1,BPR}. For information about
homological algebra  we address the reader to  \cite{HS}, while a
sounder background on the theory of exact sequence of Banach
spaces can be found in \cite{CG}. Recall that a Banach space $X$
is said to have the Bounded Approximation Property (in short, BAP)
if there is a constant $C>0$ such that for every $\epsilon >0$ and
each compact subset $K$  of $X$, there is a finite rank operator
$T$ in $X$ with $\| T \| <C$, such that  $\| Tx-x \| \leq
\epsilon$, for each $x\in K$ .\\

\textbf{The push-out construction.} Let us recall the push-out
construction from the theory of exact sequences. Given two
operators $ i: Y\rightarrow X$ and $j: Y\rightarrow E$ their
push-out is the space $PO = (E \times X )/\Delta$, where $\Delta
=[ \overline{{(jy, -iy) : y \in Y}} ]$ endowed with the quotient
topology. Let $v: E \to PO$ and $w: X \to PO$ be the operators
$v(m)= [(m,0)]$ and $v(x)= [(0,x)]$. One has $vj=wi$. Moreover,
given any Banach space $W$ and operators $a: E \to W$ and $b: X
\to W$ so that $aj=bi$ there is a unique operator $u: PO \to W$
given by $u[m,x] = am + bx$ such that $uv=a$, $uw=b$ and
$\|u\|\leq \max \{\|a\|,\|b\|\}$. In particular, this means (see
\cite[1.2, 1.3]{CG}) that given an exact sequence $0 \to Y
\stackrel{i}\to X \to Z \to 0$ and an operator $j: Y \to E$ the
push-out of $i,j$ yields a commutative diagram

$$
\begin{CD}
0@>>>Y@>i>>X@>>>Z@>>>0\\
 &&@VjVV @VVwV \Vert\\
0@>>>E@>>v>PO@>>> Z@>>>0.
 \end{CD}$$

The following lemma (see \cite[1.3]{CG}) establishes the fundamental
connection of the push-out construction regarding
extension/lifting problems:

\begin{lemma}\label{critopo} In a push-out diagram,
$$
\begin{CD}
0@>>>Y@>i>>X@>>>Z@>>>0\\
 &&@VjVV @VVwV \Vert\\
0@>>>E@>>v>PO@>>> Z@>>>0.
 \end{CD}$$
the lower sequence splits if and only if $j$ extends to $X$; that
is, there is an operator $T:X\to E$ such that $Ti=j$.
\end{lemma}

\section{Locally splitting vs. $\mathcal{A}$-splitting}

Kalton introduced in \cite{K} the notion of locally splitting for
an exact sequence as follows (see also Fakhoury \cite{Fa}):

\adef An exact sequence $0\rightarrow Y\overset{i}{%
\rightarrow }X\overset{q}{\rightarrow }Z\rightarrow 0$ is said to
locally split if there exists a constant $\lambda>0$ such that for
every finite dimensional subspace $E \subset X$ there exists a
linear continuous operator $T_{E}:E \to Y$ such that
${T_E}_{|E\cap Y}=id_{E\cap Y}$ and  $\|T_{E}\|\leq \lambda$. We will also say that $Y$ is locally complemented in $X$ when the corresponding exact sequence $0\rightarrow Y\overset{i}{%
\rightarrow }X\overset{q}{\rightarrow }X/Y\rightarrow 0$ locally
splits.  \zdef

The Principle of Local Reflexivity of Lindenstrauss and Rosenthal
\cite[Thm. 3.1.]{LR}, says that every Banach space is locally
complemented in its bidual. Also, it is well known that every
Banach space is locally complemented in its ultrapowers. Other
results about local complementation are the following:

If $Y$ is an $\mathcal{L}_{\infty }-$space and $Y\subset X$ (or
$X/Y$ is an $\mathcal{L}_{1}-$space) then $Y$ is locally
complemented in $X$ (see \cite{K}).

Thus, local splitting corresponds to uniform $\mathfrak
F$-splitting. One moreover has:

\begin{theorem}\label{locallysplit}(Mainly Kalton-Fakhoury) Let $0\rightarrow Y\rightarrow X \rightarrow Z\rightarrow 0$
be an exact sequence. The following are equivalent:
\begin{itemize}
\item[(1)]  The sequence locally splits.
 \item[(2)]  The sequence uniformly ${\mathfrak F}$-splits.
  \item[(3)]  The sequence uniformly ${\mathfrak F}$-litfs.
  \item[(4)]  The sequence ${\mathfrak K}$-splits. Equivalently,
  uniformly ${\mathfrak K}$-splits.
\item[(5)]  The sequence $ {\mathfrak W}$-splits. Equivalently,
uniformly ${\mathfrak W}$-splits \item[(6)] The sequence
$\texttt{A}$-splits. Equivalently, uniformly $\texttt{A}$-splits.
\item[(7)] The sequence $\texttt{A}$-lifts. Equivalently,
uniformly $\texttt{A}$-lifts.

If, moreover, $Z$ has the BAP then assertions (1) to (7)  are also
equivalent to
 \item[(8)] The  sequence $\mathfrak{K}$-lifts.  Equivalently,
  uniformly ${\mathfrak K}$-lifts.
\end{itemize}
\end{theorem}
\begin{proof}  The equivalences (1) to (4) were proved by Kalton \cite{K} (and by Fakhoury \cite[III]{Fa} with other terminology). Kalton also shows
\cite[Thm. 3.5]{K} that a sequence locally splits if and only if
its dual sequence splits. From this (5) can be easily derived: on
one hand, it is clear that $ {\mathfrak W}$-splitting implies $
{\mathfrak K}$-splitting; on the other hand, if the sequence
$0\rightarrow Y\stackrel{i}\rightarrow X \rightarrow Z\rightarrow
0$ locally splits, its dual sequence splits, hence also its bidual
sequence $0\rightarrow Y^{**}\stackrel{i^{**}}\rightarrow X^{**}
\rightarrow Z^{**}\rightarrow 0$ splits. Let $p: X^{**}\to Y^{**}$
be a linear continuous projection. If $t: Y\to E$ is a weakly
compact operator, its bidual $t^{**}$ still is weakly compact and
$E$-valued. Thus, $t^{**}p: X^{**}\to E$ is a weakly compact
extension of $t$ and the sequence $\mathfrak W$-splits (see also
\cite[3.1]{Fa}). The equivalence with (6) is also simple: since
every approximable operator is compact, (4) implies (6). And since
$\texttt{A}$ is closed, $\texttt{A}$-splitting implies uniform
$\texttt{A}$-splitting, hence uniform $\mathfrak {F}$-splitting,
which is (2).

Observe now that $\mathfrak K$-lifting is not in the list, so the
equivalence with (7) is new and has to be proved. That (7) implies
(2) is clear, so we show that (2) implies (7). Assertion (7)
exactly amounts showing that given any Banach space $V$ the
induced sequence
$$
\begin{CD}
0@>>>\texttt{A}(V,Y)@>i_{\circ}>>\texttt{A}(V,X)@>q_{\circ
}>>\texttt{A}(V,Z)@>>>0,
\end{CD}$$
in which $i_{\circ}$ (resp. $ q_{\circ}$) denote the operators
``left-composition with $i$" (resp. with $q$), is also exact.
Thus, assume (2). That the sequence uniformly $\mathfrak F$-lifts
implies that, for each Banach space $V$, the sequence of normed
spaces
$$
%\begin{equation}
\begin{CD}
0@>>>\mathfrak{F}(V,Y)@>i_{\circ}>>\mathfrak{F}(V,X)@>q_{\circ }>>\mathfrak{F}(V,Z)@>>>0  \\
\end{CD}
$$
%\end{equation}
is topologically exact (i.e.; $q_{\circ}$ is an open map). We show
that the sequence formed by the completion of those spaces,
namely,
$$
\begin{CD}
0@>>>\texttt{A}(V,Y)@>>>\texttt{A}(V,X)@>>>\texttt{A}(V,Z)@>>>0\\
\end{CD}$$
is also an exact sequence. Let $j: \mathfrak F(V,Y) \to
\texttt{A}(V,Y)$ denote the canonical embedding. Making the
push-out of the couple $j, i_{\circ}$ one gets the commutative
diagram
$$
\begin{CD}
0@>>>\mathfrak{F}(V,Y)@>i_{\circ}>>\mathfrak{F}(V,X)@>q_{\circ}>>\mathfrak{F}(V,Z)@>>>0\\
 &&@VjVV @VVV \Vert\\
0@>>>\texttt{A}(V,Y)@>>>PO@>>> \mathfrak{F}(V,Z)@>>>0.\\
 \end{CD}$$

The universal property of the push-out applied to the operators
$i_\circ: \texttt{A}(V,Y) \to \texttt{A}(V,X)$ and $j:
\mathfrak{F}(V,X) \to \texttt{A}(V,X)$ (the canonical embedding)
yields the existence of the operator $u: PO \to \texttt{A}(V,X)$
given by $u[a,F]= i_\circ a + F$ and a commutative diagram
$$
\begin{CD}
0@>>>\mathfrak{F}(V,Y)@>i_{\circ}>>\mathfrak{F}(V,X)@>q_{\circ}>>\mathfrak{F}(V,Z)@>>>0\\
 &&@VjVV @VVwV \Vert\\
0@>>>\texttt{A}(V,Y)@>>v>PO@>>> \mathfrak{F}(V,Z)@>>>0\\
 &&\Vert && @VVuV @VV{\overline {u}}V \\
0@>>>\texttt{A}(V,Y)@>>{i_\circ}> \texttt{A}(V,Y) @>>>
\diamondsuit@>>>0.
 \end{CD}$$
The operator $u$ has dense range since $uw$ is the canonical
dense-range embedding $j: \mathfrak{F}(V,X) \to \texttt{A}(V,X)$;
so $\overline {u}$ must also have dense range. And since
$\diamondsuit$ is complete, it
must be the completion $\texttt{A}(V,Z)$ of $\mathfrak{F}(V,Z)$.\\

The equivalence with (8) follows from this since when  $Z$ has the
BAP then for every Banach space $V$ one has the identity
$\mathfrak K(V, Z) = \texttt{A}(V, Z)$.
\end{proof}

The following proposition explains why $\mathfrak K$-lifting is
not in the list of Theorem \ref{locallysplit}.

\begin{proposition}\label{bap} A separable Banach space $Z$ has the BAP  if and only if for every exact
sequence $0\rightarrow Y \rightarrow X \rightarrow Z \to 0$
$\mathfrak K$-spliting and ${\mathfrak{K}}$-litfing are
equivalent.
\end{proposition}
\begin{proof} Recall that James and Lindenstrauss \cite[1.d.3]{LT} proved that if $Z$ is a
separable Banach space there exists a separable Banach space $X$
such that both $X$ and $X^{\ast \ast }$ have basis and moreover
$X^{**}/X= Z$. So there is an exact sequence
$$
\begin{CD}
0@>>>X@>>>X^{\ast \ast }@>>>Z@>>>0\\
\end{CD}
$$ which locally splits by the Principle of Local Reflexivity \cite[Thm. 3.1]{LR}. If
$Z$ fails the BAP, there exists a compact non-approximable
operator $K\in \mathfrak{K}(Z,Z)$ \cite[5.3]{DF}. If there would
exist a compact lifting $\widetilde{K}:Z\rightarrow X^{\ast \ast
}$ of $K$ then  $\widetilde{K}$, hence also $K$, would be
approximable.
\end{proof}

It therefore follows that $\mathfrak K$-splitting and $\mathfrak
K$-lifting are not equivalent notions. It is not hard to see that
$\mathfrak K$-lifting and ${\mathfrak {W}}$-lifting are also
non-equivalent, even in the presence of the BAP: indeed, it is
shown in \cite{KA} (see also \cite[Thm. 2.3.]{CCKY}) the existence of a
nontrivial exact sequence
$$
\begin{CD}
0@>>>C[0,1]@>>>X@>>>\ell_2@>>>0;
\end{CD}$$
or else, it is obtained in \cite{JL} the existence of nontrivial
exact sequences
$$
\begin{CD}
0@>>>c_0@>>>X@>>>\ell_2(I)@>>>0.
\end{CD}$$

In both cases, recall --see also \cite{K, LR}-- that any exact
sequence $0\rightarrow Y\rightarrow X \rightarrow Z\rightarrow 0$
in which $Y$ is an $\mathcal L_\infty$-space locally splits.

\section{Extension and lifting of polynomials and holomorphic mappings}
For the  polynomial version  of Theorem \ref{locallysplit},  the basic idea is to reduce the problem for polynomials and
holomorphic maps to a problem on linear operators.\\

We denote by $\mathcal{P}(^{n}X)$ the Banach space of all
continuous $n$-homogeneous polynomials on $X$, where the norm is
given by $\|P\| = \sup\{P (x):\|x\| \leq 1\}$. The
natural-isometric predual of $\mathcal{P}(^{n}X)$ is the
projective symmetric tensor product $\hat{\otimes }^{n,s}_{\pi_{s}
}X$ and  the mapping $x \to x^n=\otimes^nx$ is a ``universal"
continuous $n$-homogeneous polynomial on $X$: for every $P\in
\mathcal{P}(^{n}X)$, there is a unique linearization $\widehat{P}
\in \mathcal{L}(\hat{\otimes }^{n,s}_{\pi_{s} }X)$ with the same
norm, such that $P(x)=\widehat{P} (x^n)$ for every $x\in X$  (see
\cite[2.2]{F}). Similarly, we have the vector-valued  version of
the previous identity
$\mathcal{P}(^{n}X,V)=\mathcal{L}(\hat{\otimes }^{n,s}_{\pi_{s}
}X,V)$.

Let $X$ be a complex Banach space. We denote by   $\mathcal{H}_{b} (X)$ the Fr\'echet space of all  holomorphic  mappings on $X$ that are bounded
on the bounded subsets of $X$. Let $\mathcal{H}_{\mathfrak {K}}(X)$ (resp. $\mathcal{H}_{\mathfrak {W}}(X)$) be
the space of the holomorphic  bounded functions in $X$ that are compact (respectively $w$-compact).

Let $f$ be a holomorphic function on $X$. The Taylor series
\begin{equation*}
f(x)\sim \sum\limits_{n=1}^{\infty }\frac{d^{n}f(0)}{n!}(x) \qquad
(d^{n}f(0)\in \mathcal{P}(^{n}X)\simeq \mathcal{L}_{s}(^{n}X)).
\end{equation*}%
decomposes $f$ as a formal sum $\sum_n P_n$ where $P_n\in
\mathcal{P}(^{n}X)$. It is well known that $f
\in\mathcal{H}_{\mathfrak {K}}(X)$ if and only if $P_n\in
\mathcal{P}_{\mathfrak {K}}(^{n}X)$ for all $n$ (same  for
$\mathcal{H}_{\mathfrak {W}}(X)$), see \cite[Prop. 5]{GG1} or
\cite{AS}.

The definitions and general  properties of projective and
injective tensor product spaces can be found in \cite{DF}, and \cite{F} for the symmetric tensor product. In
particular, the injective tensor product $\check{\otimes
}_{\varepsilon }$ has deep connections with vector functions
spaces since it often occurs that injective tensor product spaces
can be represented as vector function spaces. Some examples for
this assertion: $X\check{\otimes }_{\varepsilon }Y$ is always a
subspace of $ \mathfrak{K}_{w^*}(X^*,Y)$ --the space of compact
weak*-continuous operators-- with  equality when $X^*$ or $Y$ have
BAP; also, the space $C(K,X)$ of continuous $X$-valued functions
on a compact space $K$ coincides with $C(K)\check{\otimes
}_{\varepsilon }X$; and the space
$\mathcal{P}_{\texttt{A}}(^{n}Y,X)$ coincides with
$(\check{\otimes }^n_{\varepsilon,{s} }Y^*)\check{\otimes
}_{\varepsilon }X$ (\cite[5.3]{DF}). Further examples of distribution spaces in
locally vector spaces can be found in \cite{B,Ka}. Especially
interesting for us is the identification
$$(V\check{\otimes }_{\varepsilon }Z)^*=\mathcal{I}(V,Z^*)$$
of the dual of the injective tensor product as the space
$\mathcal{I}(V,Z^*)$ of integral operators, see
\cite[\S3. \S4.]{DF} and the identification
$$(V\hat{\otimes }_{\pi }Z)^*
=\mathfrak{L}(V,Z^*)$$ of the dual of the projective tensor
product as the space $\mathfrak{L}(V,Z^*)$ of all operators.

It is well-known that the tensorization $V\check{\otimes
}_{\varepsilon} {-} $ of an exact sequence is not necessarily
exact (see in Section 7 the notions of left-exact and exact
functor). Kaballo \cite{Ka} defines an exact sequence
$0\rightarrow Y \rightarrow X \rightarrow Z \rightarrow 0$ of
locally convex spaces to be an $(\varepsilon L)$-triple when, for
every Banach space $V$, the tensorized sequence
$$\begin{CD} 0 @>>> V\check{\otimes }_{\varepsilon  }Y @>>>
V\check{\otimes }_{\varepsilon  }X @>>> V\check{\otimes
}_{\varepsilon  }Z @>>> 0
\end{CD}$$is exact. In the category of Banach spaces one has:

\begin{proposition}\label{approximation}  Let $0\rightarrow Y\rightarrow X \rightarrow Z\rightarrow 0$
be an exact sequence. The following are equivalent:
\begin{enumerate}
\item The sequence is an $(\varepsilon L)$-triple.

\item The sequence locally splits.

\item For every Banach space $V$, the  sequence

$$
\begin{CD} 0 @>>> V\check{\otimes }_{\varepsilon  }Y @>>>
V\check{\otimes }_{\varepsilon  }X @>>> V\check{\otimes
}_{\varepsilon  }Z @>>> 0\\
\end{CD}
$$
is exact and locally splits.

\item For every Banach space $V$, the sequence
$$\begin{CD} 0 @>>> V\hat{\otimes }_{\pi  }Y @>>>
V\hat{\otimes }_{\pi  }X @>>> V\hat{\otimes }_{\pi  }Z @>>> 0
\end{CD}$$ is exact and locally splits.
\end{enumerate}
\end{proposition}
 \begin{proof} The equivalence between (1) and (2) was
 obtained by Kaballo \cite[Thm. 2.2]{Ka}.  To see that (1) implies (3) it is enough to check
 that the adjoint operator $(Id{\otimes }q)^*:(V\check{\otimes }_{\varepsilon  }Z)^* \rightarrow
(V\check{\otimes }_{\varepsilon  }X)^*$  admits a linear
continuous projection. Let  $r$ be a projection for $q^*$; i.e.,
$r\circ q^*=Id_{Z^*}$. Under the identification of the dual of the  injective tensor space with the space of integral operators
the operator $(Id{\otimes }q)^*$ becomes left-composition with
$q^*$. Thus, the operator $R:\mathcal{I}(V,X^*) \rightarrow
\mathcal{I}(V,Z^*)$ given by  $R(I) = r \circ {I}$ induces a
projection through $(Id{\otimes }q)^*$. The case of the projective
tensor product is analogous using the corresponding identification
$(V\hat{\otimes }_{\pi }Z)^* =\mathcal{L}(V,Z^*)$. That (3),(4)
imply (2) is clear since $ V\check{\otimes }_{\varepsilon  }Z =
V\hat{\otimes }_{\pi }Z =Z$ when $V=\mathbb{K}$. \end{proof}

A few more facts about the so-called Aron-Berner extension for
polynomials will be required. If $Y$ is a locally complemented
subspace of $X$, it is then clear that there exists  a linear
continuous section $s:Y^* \to X^*$  extending operators (i.e.;
$i^*s=Id_{Y^*}$, \cite[Th. 3.5]{K}). The operator $s$ induce
--just using induction on $n$-- continuous linear maps
$$
AB: \mathcal{P}(^nY)\to \mathcal{P}(^nX).\\
$$

This is the well-known Aron-Berner extension. We will use the
notation $AB(P)=\overline{P}$. Different descriptions of the
Aron-Berner extension can be seen in \cite{AB,CGV,GJLl,Z}. The
operator  $AB$  is a section for the restriction operator
$R:\mathcal{P}(^nX)\to \mathcal{P}(^nY)$. Thus, the following
exact sequence locally splits
$$
\begin{CD} 0 @>>>\hat{\otimes }_{\pi,s }^{n}Y@>{\otimes i}>>\hat{\otimes }_{\pi,s }^{n}X@>>>
(\hat{\otimes}_{\pi, s}^{n}X)/(\hat{\otimes }_{\pi,s }^{n}Y)@>>>
0.\\
\end{CD}$$\

About what types of polynomials are preserved by the Aron-Berner
extension, it is clear that polynomials of finite type,
approximable, compact and weakly compact polynomials are preserved (see \cite{CL,Z} for more classes of polynomials).\\

We now define the Aron-Berner extension for vector-valued
polynomials. Let $Y,V$ be Banach spaces. If $\phi \in V^*$, the
operator
$$
AB: \mathcal{P}(^nY;V)\to \mathcal{P}(^nX,V^{**})
$$
is defined by composition  $AB(P)(x)(\phi)=(\overline{\phi \circ
P})(x)$ (see \cite{CGV,CL,Z}). In general, $AB$ does not take its
values in $V$. It is clear that if $P\in \mathcal{P}(^nY;V)$ has
weakly compact associated linear operator $T_{P}:Y\to
\mathcal{P}(^{n-1}Y;V) $ then $AB(P)$ is $V$-valued (see
\cite[Section 2.3]{CL}). Since the classes  of finite type,
approximable, compact and  weakly compact polynomials all have
weakly compact associated linear operator, their respective
Aron-Berner extensions are $V$-valued (see \cite{CL} for details
and the
consideration of other classes of polynomials).\\

We are thus ready to obtain extension/lifting theorems for
polynomials and holomorphic functions.

\begin{theorem}\label{ls} Let $0\rightarrow Y\overset{i}{\rightarrow }
X\overset{q}{\rightarrow }Z\rightarrow 0 $ be an exact sequence of
Banach spaces. The following are equivalent:
\begin{enumerate}
\item The sequence locally splits.

\item  The sequence $\mathcal{P}_{\texttt A}$-splits.
 \item The sequence  $\mathcal{P}_{\texttt A}$-lifts.
\item  The sequence $\mathcal{P}_{\mathfrak {K}}$-splits, \item
The sequence $\mathcal{H}_{\mathfrak {K}}$-splits. \item The
sequence $\mathcal{P}_{\mathfrak {W}}$-splits \item  The sequence
$\mathcal{H}_{\mathfrak
 {W}}$-splits.

If, moreover, $Z$ has the BAP, they are also equivalent to \item
The sequence  $\mathcal{P}_{\mathfrak {K}}$-lifts. \item The
sequence $\mathcal{H}_{\mathfrak {K}}$-lifts.
\end{enumerate}
\end{theorem}
\begin{proof} $(1)\Rightarrow (3)$. To lift approximable
polynomials it is enough to apply Proposition \ref{approximation}
and  tensorize  with  $\mathcal{P}_{\texttt A}(V)$ since
$\mathcal{P}_{\texttt A}(^{n}V,X )=\mathcal{P}_{\texttt A}(^{n}V)
\check {\otimes }_{\varepsilon }X$  (\cite[5.3]{DF}). \\

$(1)\Rightarrow (2)$ Let $P\in \mathcal{P}_{f}(^nY;V)$ be a
polynomials of finite type, say $P=\sum_{i=1}^{m}\phi^n_{i}y_{i}$
where $y_{i}\in Y$ and $\phi_{i}\in Y^*$. Then
$AB(P)=\sum_{i=1}^{m}s(\phi)^n_{i}y_{i}$. It is follows that
$AB:\mathcal{P}_{f}(^nY;V)\to \mathcal{P}_{f}(^nX;V)$ is a
complemented embedding. Using the arguments of the proof of
Theorem  \ref{locallysplit}, the continuity of $AB$ and the
completeness of $V$ yield (2).\\

$(1)\Rightarrow (4) $ and $(6)$. For the case of
$\mathcal{P}_{\mathfrak {K}}$-splitting and
$\mathcal{P}_{\mathfrak {W}}$-splitting, let $P\in $
$\mathcal{P}(^{n}X,V)$ and denote by $\widehat{P}\in $
$\mathcal{L}(\hat{\otimes }_{\pi,s }^{n}X,V)$ the linearization of
$P$ transforming compact polynomials (resp. $w$-compact) into
compact operators (resp. $w$-compact) and viceversa (see
\cite[Lemma 4.1]{Ry}). Since the sequence locally splits, the
Aron-Berner extension provides the locally splitting of the
sequence (see \cite[2.6]{CGV})
$$
\begin{CD} 0 @>>>\hat{\otimes }_{\pi,s }^{n}Y@>{\otimes i}>>\hat{\otimes }_{\pi,s }^{n}X@>>>
Q=(\hat{\otimes}_{\pi, s}^{n}X)/(\hat{\otimes }_{\pi,s }^{n}Y)@>>>
0.
\end{CD}$$
For a given polynomial $P\in \mathcal{P}_{\mathfrak {K}}(^{n}Y,V)$
(resp. $\mathcal{P}_{\mathfrak {W}}(^{n}Y,V)$), one gets
 that the linearization of $\widehat{P}$ can be extended as in the diagram
  $$
\begin{CD}
0@>>>{\hat{\otimes}_{\pi,s }^{n}Y}@>{\otimes i}>>{\hat{\otimes
}_{\pi,s }^{n}X}@>>>Q@>>>0\\
 &&@V\widehat{P}VV @V{\overline{\widehat{P}}}VV \\
&& V &=&V
 \end{CD}
 $$since the tensorized sequence ${\mathfrak {K}}$-splits (resp.
${\mathfrak {W}}$-splits) by Theorem \ref{locallysplit}.\\

$(1)\Rightarrow (5) $ and $(7)$. Let us show now the case of
holomorphic functions. Let $f\in \mathcal{H}_{\mathfrak {K}}(Y,V)$
and let $f(x)=\sum\limits_{k}\frac{d^{k}f(0)}{k!}(x)$ be its
Taylor series, where $d^{k}f(0)\in \mathcal{P}_{\mathfrak
{K}}(^{n}X,V)$ (see \cite{AS,D}). Then we may define the extension
operator $ \Phi :\mathcal{H}_{\mathfrak {K}}(Y,V)\rightarrow
\mathcal{H}_{\mathfrak {K}}(X,V)$ as
\begin{equation*}
\Phi(f)=\overline{f}=\sum\limits_{k}\frac{\overline{d^{k}f}(0)}{k!}.
\end{equation*}
The convergence follows from \cite[Lemma 3.1]{CL} or \cite{Z} (see
\cite{AB,AS,D} for the scalar case). The same argument works for
the case $\mathcal{H}_{\mathfrak {W}}(Y,V)$ (see  \cite{AS,CL}).
To obtain the converse, we just recall that the operators are
exactly the polynomials of degree $1$. We may also easily check
that the map $s:Y^{\ast }\rightarrow X^{\ast }$ given by
$s(y^{_{\ast }})=d(\Phi (y^{\ast }))(0)$ is a section for $i^{\ast
}:X^{\ast }\rightarrow Y^{\ast }$.\\

$(1)\Rightarrow (8)$. If $Z$ has the BAP then one gets
$\mathcal{P}_{\mathfrak {K}}(^nV,Z)=\mathcal{P}_{A}(^nV,Z)$
(\cite[5.3]{DF}), and the result follows from $(3)$.
\end{proof}

Given a closed injective operator ideal $\mathcal U \subseteq
\mathcal W$ the theorem above can be extended to the classes
$\mathcal P_{\mathcal U}$ using the polynomial factorization given
in \cite[Cor. 5]{GG2}. See the comment after \cite[Cor. 2.7.]{CL}
for details. Taking this into account, Theorem \ref{ls} yields.

\begin{corollary}\label{KW^*} Let $0\rightarrow Y\overset{i}{\rightarrow }
X\overset{q}{\rightarrow }Z\rightarrow 0 $ be an exact sequence of
Banach spaces. The following are equivalent:
\begin{enumerate}
\item The sequence locally splits

\item For every dual Banach space $V^*$ with the BAP the functor
$\mathcal{P}_{\mathfrak{K}w^*}(^nV^*,-) $ transforms it into an
exact sequence.

\item For every dual Banach space $V^*$ with the BAP the functor
$\mathcal{H}_{\mathfrak{K}w^*}(V^*,-) $ transforms it into an
exact sequence.
\end{enumerate}
\end{corollary}
\begin{proof} It is well-known \cite{DF, F} that if $V^*$ has the
BAP then also  $( \check {\otimes }^n_{s,\varepsilon } {V} )^*$
has the BAP. The result now follows from the identification $$(
\check {\otimes }^n_{s,\varepsilon } {V} ) \check {\otimes
}_{\varepsilon } {{-}} =
\mathcal{P}_{\mathfrak{K}w^*}(^nV^*,{{-}}).$$

\end{proof}

\section{Characterization of $\mathcal{L}_{\infty}$-spaces and $\mathcal{L}_{1}$-spaces}

A Banach space $X$ is said to be an $\mathcal L_{\infty,
\lambda}$-space (resp. $\mathcal L_{1, \lambda}$-space) if every
finite dimensional subspace $E\subset X$ is contained in a finite
dimensional subspace $F\subset X$ that is $\lambda$-isomorphic to
$\ell_\infty^{dim F}$ (resp. $\ell_1^{dim F}$). If no reference to
$\lambda$ is necessary we will simply say that $X$ is an $\mathcal
L_\infty$ (resp. $\mathcal L_1$-space). Lindenstrauss and
Rosenthal characterize in \cite[Thm. 4.1.]{LR} the
$\mathcal{L}_{\infty}$ (resp. $\mathcal{L}_{1}$)-spaces through
the $\mathfrak{K}$-splitting (resp. ${\mathfrak {K}}$-lifting) of
exact sequences (equivalences $(1) \Leftrightarrow (2)$ in
Theorems \ref{eleinf} and \ref{eleuno} below), while Gonz\'{a}lez
and Guti\'{e}rrez in \cite{GG} extend the result using
$\mathcal{P}_{\mathfrak{K}}$-splitting (resp.
$\mathcal{P}_{\mathfrak{K}}$-lifting). We provide next a unified
approach to these results and several generalizations.

Indeed, $\mathcal L_\infty$ and $\mathcal L_1$-spaces are
necessarily involved in our study of extension/lifting of
operators. To show why, recall that a Banach space $X$ is said to
be injective if every exact sequence $0\rightarrow X \rightarrow
\diamondsuit\rightarrow \spadesuit\rightarrow 0$ splits. It is on
the other hand well-known that a Banach space is an $\mathcal
L_\infty$-space if and only if its bidual space is injective \cite[p. 335]{LR}. One
therefore has the following characterization of $\mathcal
L_\infty$-spaces.

\begin{theorem}\label{eleinf} Let $Y$ be a Banach space. The following conditions are
equivalent:
\begin{enumerate}
\item $Y$ is an $\mathcal L_\infty$-space.

\item Every exact sequence $0\rightarrow Y \rightarrow
\diamondsuit\rightarrow \spadesuit\rightarrow 0$ locally splits.

\item Every exact sequence $0\rightarrow Y \rightarrow
\diamondsuit \rightarrow \spadesuit\rightarrow 0$ $\mathcal
A$-lifts for any of the choices $\mathcal A = \mathfrak K$,
$\mathcal{P}_{\mathfrak{K}}$ or $\mathcal{H}_{\mathfrak{K}}$.

\item Each of the functors $Y\check{\otimes }_{\varepsilon }{-} $,
$\mathfrak{K}_{w^*}(Y^*,-)$,
$\mathcal{P}_{\mathfrak{K}w^*}(^nY^*,-) $ or
$\mathcal{H}_{\mathfrak{K}w^*}(Y^*,-) $ transform exact sequences
into exact sequences.

\end{enumerate}
\end{theorem} \begin{proof} The equivalence between $(1)$ and $(2)$ is thus clear
from the comments before the statement of the theorem and Kalton's
characterization of locally splitting.

$(1)\Rightarrow (3)$ Let $K:Z\to \spadesuit $ be a compact
operator. Take an index set $\Gamma$ for which there exists a linear
continuous surjection $Q: \ell_1(\Gamma)\to Z$. Form then the
commutative diagram
$$
\begin{CD}
0@>>>\ker Q@>i>>\ell_1(\Gamma)@>Q>>Z@>>>0\\
 &&@VV{K_1}V @VV{K_0}V @VV{K}V\\
 0@>>> Y @>j>> \diamondsuit @>q>> \spadesuit @>>> 0\\
 \end{CD}
$$ where we can assume that $K_0$ is a compact lifting of $KQ$ with $\|K_0\|\leq (1+\varepsilon)\|K\|$, see \cite[Thm. 4.1.]{LR}.
Thus, $K_1$ must also be compact. Consider then a compact
extension $\widetilde{K}_{1}:\ell_1(\Gamma)\to Y$ of $K_1$, which
yields  a compact operator
$j\widetilde{K}_{1}-K_0:\ell_1(\Gamma)\to \diamondsuit$. Since
$(j\widetilde{K}_{1}-K_0)i=0$, the operator
$j\widetilde{K}_{1}-K_0$ must factorize through a (necessarily
compact) operator $\widetilde{K}:Z\to \diamondsuit$ which is the
desired lifting of $K$ we were looking for.

The exactness of $Y\check{\otimes }_{\varepsilon }{-} $ can be
found in \cite[Thm. 1.5]{Ka}; see also \cite[p. 307]{DF}). For the
remaining assertions  one just need to recall (\cite[p. 68]{Ry1}
and \cite[3.1]{F}) that when $Y$ is an $\mathcal L_\infty$ space
then also  $\check{\otimes }^n_{\varepsilon,s }{Y} $ is an
$\mathcal{L}_{\infty}$-space and that the equality
$\mathcal{P}_{\mathfrak{K}w^*}(^nY^*,{-})=( \check {\otimes
}^n_{s,\varepsilon } {Y} ) \check {\otimes }_{\varepsilon } {-} $
holds since $Y$ is complemented in
$\check{\otimes }^n_{\varepsilon,s }{Y} $.\\

For the holomorphic case, we consider that the decomposition of
$f$ as a formal sum $\sum_n P_n$ where $P_n\in
\mathcal{P}_{\mathfrak{K}w^*}(^nY^*,{-})$ ( \cite[Prop. 5]{GG1}.) The result follows from the proof of Theorem \ref{ls} .\\

$(3)\Rightarrow (1)$ Given a compact operator $K:X\to Y$, consider a commutative diagram

$$
\begin{CD}
0@>>>X@>i>>\diamondsuit@>Q>> \spadesuit@>>>0\\
 &&@VV{K}V @VV{K_0}V @VV{K_1}V\\
 0@>>> Y @>j>>\ell_{\infty}(\Gamma) @>q>>\ell_{\infty}(\Gamma)/Y @>>> 0,
 \end{CD}$$
where $K_0,K_1$ can be chosen to be compact. A similar argument as above works.\\

$(4)$ This follows from the Proposition \ref{approximation} and Corollary \ref{KW^*}.

\end{proof}

\begin{remark} \
\begin{enumerate}
\item[a)]  The case (3) for $\mathfrak K$ (resp. (2) for
$\mathfrak W$) appear in Domanski \cite[Thm. 4]{dom}, equivalence
between (b.i) and (b.ii) (resp. between (a.i) and (a.ii)).\

\item[b)] Gonz\'alez and Guti\'errez show in \cite[Remark 1]{GG}
that  $\mathcal{P}_{\mathfrak{K}}(^n{-},Y)$ does not necessarily
transforms exact sequences into exact sequences, even when $Y$ is
an $\mathcal L_\infty$-space.
\end{enumerate}
\end{remark}

Dually, it is well-known  that a Banach space is is an $\mathcal
L_1$-space if and only if its dual space is injective (\cite{LR}).
One therefore has.

\begin{theorem}\label{eleuno} Let $Z$ be a Banach space. The following conditions are equivalent
\begin{enumerate}
\item $Z$ is an $\mathcal L_1$-space.\

\item Every exact sequence $0\rightarrow \spadesuit \rightarrow
\diamondsuit \rightarrow Z \rightarrow 0$ locally splits.

\item Every exact sequence $0\rightarrow \spadesuit \rightarrow
\diamondsuit \rightarrow Z \rightarrow 0$ $\mathcal A$-lifts for
any of the choices $\mathcal A = \mathfrak K$,
$\mathcal{P}_{\mathfrak{K}}$  or $\mathcal{H}_{\mathfrak{K}}$.

\item Every exact sequence $0\rightarrow R \rightarrow
\diamondsuit \rightarrow Z \rightarrow 0$ in which $R$ is
reflexive splits.

\item Each of the functors $Z\hat{\otimes }_{\pi }{-} $ ,
$\mathfrak{K}(Z,-)$, $\mathcal{P}_{\mathfrak{K}}(^nZ,-) $ or
$\mathcal{H}_{\mathfrak{K}}(Z,-) $ transform exact sequences into
exact sequences.

\end{enumerate}
\end{theorem}
\begin{proof}
The equivalence between (1) and (2) is clear, and
since $\mathcal L_1$ spaces have the BAP, the equivalence with (3)
is also clear by Theorem \ref{ls}. (4) is consequence of the $\mathfrak W$-splitting
and the Davis-Figiel-Johnson-Pe\l czy\'nski factorization of
weakly compact operators \cite[9.6]{DF}.  The equivalence $(1)$ and $(5)$ follows
from \cite[3.9]{DF}. The proof for
$\mathcal{P}_{\mathfrak{K}}(^nZ,-) $ follows from the equality
$\mathcal{P}_{\mathfrak{K}}(^{n}Z,-)=((\check{\otimes
}^n_{\varepsilon,{s} }Z^*)\check{\otimes }_{\varepsilon }-)$
and from the fact that $(\check{\otimes }^n_{\varepsilon,{s}
}Z^*)$ is an $\mathcal L_{\infty}$-space (\cite[p. 68]{Ry1} and \cite[3.1]{F}).
\end{proof}

Observe that (4) has no analogue for $\mathcal L_\infty$-spaces
since $\mathcal {L}_{\infty}$-spaces are not characterized by
${\mathfrak{W}}$-lifting, as previous examples show. $\mathcal
L_{1}$-spaces are not characterized by $\mathfrak W$-lifting
either: given a sequence $ 0 \to \ker q \to \ell_1 \stackrel{q}\to
L_{1}[0,1] \to 0$, no embedding $\ell_2\rightarrow L_{1}[0,1]$ can
be lifted through $q$ since every operator $\ell_2\to \ell_1$ must
be compact.

\section{Extension vs. lifting for operator ideals}

The difference between $\mathfrak K$-splitting and $\mathfrak
K$-lifting was remarked by Theorem \ref{locallysplit}.  In the proof of Theorem \ref{eleinf} commutative diagrams
such as
$$
\begin{CD}
0@>>>Y_1@>i>>X_1@>Q>>Z_1@>>>0\\
 &&@VVTV @VVSV @VVRV\\
 0@>>> Y @>j>> X @>q>> Z @>>> 0\\
 \end{CD}
$$appeared. It is well known that for such diagrams
\cite[1.2, 1.3]{CG} --see also \cite[Prop. 1]{dom}-- the operator
$R$ can be lifted to an operator $Z_1 \to X$ if and only if $T$
can be extended to an operator $X_1\to Y$. In the proof of Theorem
\ref{eleinf} there was moreover exhibited a certain symmetry
between ``$R$ admits a compact lifting" and ``$T$ admits a compact
extension". One could ask about the exact nature of such
extension/lifting behaviour and if the same happens for other
operator ideals.

\adef An operator ideal $\mathfrak A$ will be called balanced when
given a commutative diagram
$$
\begin{CD}
0@>>> \ker q@>>>\ell_1(\Gamma)@>q>>Z@>>>0\\
 &&@V{\varphi}VV @VVV \Vert\\
0@>>>Y@>>>X@>>> Z@>>>0\\
 &&\Vert && @VVV @VV{\psi}V \\
 0@>>> Y @>j>> \ell_{\infty}(\Lambda) @>Q>> \ell_{\infty}(\Lambda)/Y @>>>
 0\end{CD}$$ the operator $\varphi$ can be chosen in $\mathfrak A$ if and only
 if $\psi$ can be chosen in $\mathfrak A$.\zdef

Observe that this does not mean that $\psi$ must be in $\mathfrak
A$ whenever $\varphi$ is in $\mathfrak A$; it rather means that
whenever $\varphi$ is in $\mathfrak A$ the operator $j\varphi$
admits an extension $\ell_1(\Gamma) \to \ell_\infty(\Lambda)$
whose induced operator $Z\to \ell_\infty(\Lambda)/Y$ is in
$\mathfrak A$ and viceversa: whenever $\psi$ is in $\mathfrak A$
the operator $\psi q$ admits a lifting $\ell_1(\Gamma) \to
\ell_\infty(\Lambda)$ whose restriction $\ker q \to Y$ is in
$\mathfrak A$.

For reasons hidden in the homological roots of the problem (see
Section \ref{homo}) this notion is interesting only when the ideal
$\mathfrak A$ is injective and surjective (see \cite{Pi}; see also
Section \ref{homo}). Classical injective and surjective operator
ideals appearing in the literature (see \cite{Pi}) are the ideals
$\mathfrak F$; $\mathfrak K$; $\mathfrak W$; $\mathfrak U$
(unconditional summing operators); $\mathfrak R$ (Rosenthal
operators) and $\mathfrak X$ (separable range operators). We
determine now their balanced character.

\begin{proposition}\label{idealbalanced}
The ideal $\mathfrak K$ is balanced. The ideals $\mathfrak W$,
$\mathfrak U$, $\mathfrak X$ and $\mathfrak R$ are not balanced.
\end{proposition}
\begin{proof}

\emph{The ideal $\mathfrak K$ of compact operators is balanced}.
Assume that $\psi$ is compact, which makes $\psi q$ also compact.
An observe that an operator $T: \ell_1(\Gamma)\to X$ is compact if
and only if for every countable part $\N\subset \Gamma$ the
restriction $T_{|\ell_1(\N)}$ is compact. Which occurs if and only
if $\{T(e_n)\}_n$ is a relatively compact set. In conclusion, that
an operator $Ti: \ell_1(\Gamma)\to X$ is compact if and only if
given any sequence of $e_i's$ there is a subsequence $(e_j)$ for
which $(T e_j)$ converges. Let $\omega$ be a continuous
(non-linear) selection for the quotient map $\ell_\infty(\Lambda)
\to \ell_\infty(\Lambda)/Y$, which exists by the Bartle-Graves
selection theorem \cite[p. 52]{CG}. The operator $\Phi:
\ell_1(\Gamma) \to \ell_\infty(\Lambda)$ defined by
$\Phi(e_i)=\omega \varphi q e_i$ is therefore compact, and its
restriction to $\ker q$ yields a compact operator $\varphi:\ker q
\to Y$ that makes commutative the diagram above. Conversely,
assume that $\varphi$ is compact. Lindenstrauss's
characterizations of $\mathcal L_\infty$ spaces \cite{L} yield a
compact extension $K: \ell_1(\Gamma) \to \ell_\infty(\Lambda)$ of
$j\varphi$. The operator $QK$ factorizes as $QK = \psi q$ with
$\psi$ compact.\\

\emph{The ideal $\mathfrak W$ of weakly compact operators is not
balanced}. Using a result of Bourgain and Pisier \cite[Thm.
2.1]{BP}, every separable Banach space $X$ can be embedded into a
separable $\mathcal L_\infty$-space $\mathcal L_\infty(X)$ in such
a way that the corresponding quotient $\mathcal L_\infty(X)/X$ has
the Schur property (namely, weakly convergent sequences are norm
convergent). Consider then an exact sequence $0 \to D \to \ell_1
\to \ell_2 \to 0$, apply the Bourgain-Pisier construction to $D$
and combine both results in a push-out diagram$$
\begin{CD}
&&0&=&0\\
&&@VVV @VVV\\
  0@>>> D @>>> \ell_1 @>>> \ell_2 @>>> 0\\
 & &@VVV @VVV \|\\
 0 @>>> \mathcal L_\infty(D)@>>> PO @>>> \ell_2 @>>>0\\
 & &@VVV @VVV\\
 & &S &=&S\\
 &&@VVV @VVV\\
 &&0&=&0\end{CD}
$$ By a standard 3-space argument (see \cite[Thm. 6.1.a]{CG}) the space
$PO$ has the Schur property, hence the sequence  $0 \to \mathcal
L_\infty(D) \to PO \to \ell_2 \to 0$ cannot split. Consider a
commutative diagram$$
\begin{CD}
0@>>> \ker q@>>>\ell_1@>q>>\ell_2@>>>0\\
 &&@V{\varphi}VV @VVV \Vert\\
0@>>>\mathcal L_\infty(D)@>>>PO@>>> \ell_2@>>>0\\
 &&\Vert && @VVV @VV{\psi}V \\
 0@>>> \mathcal L_\infty(D) @>j>> \ell_{\infty} @>Q>> \ell_{\infty}/\mathcal L_\infty(D) @>>>
 0.\end{CD}$$
The operator $\psi$ is always weakly compact. However, the
operator $\varphi$ cannot be weakly compact: since the space
$\mathcal L_\infty(D)$ has the Schur property, every weakly
compact operator with range $\mathcal L_\infty(D)$ must be
compact, and could therefore be extended to
any separable superspace. This would make the middle sequence split.\\

\emph{The ideal $\mathfrak U$ of unconditionally summing operators
is not balanced.} Recall that an operator $t: X\to Y$ belongs to
$\mathfrak U(X,Y)$ (i.e., it is unconditionally summing) if and
only if it is never an isomorphism on a copy of $c_0$. Recall from
\cite{KA}, or else \cite{CCKY}, examples of nontrivial sequences
$$\begin{CD} 0@>>> C[0,1]@>>> X @>>>c_0@>>>0\end{CD}.$$
Consider then a commutative diagram
$$
\begin{CD}
0@>>> \ker q@>>>\ell_1@>q>>c_0@>>>0\\
 &&@V{\varphi}VV @VVV \Vert\\
0@>>>C[0,1]@>>>X@>>> c_0@>>>0\\
 &&\Vert && @VVV @VV{\psi}V \\
 0@>>> C[0,1] @>j>> \ell_{\infty} @>Q>> \ell_{\infty}/C[0,1] @>>>
 0.\end{CD}$$
The operator $\varphi$ is always unconditionally summing since
$\mathfrak L(\ell_1, Y) = \mathfrak U(\ell_1, Y)$ for every Banach
space $Y$. On the other hand, every operator on $c_0$ is either
weakly compact or an isomorphism on a copy of $c_0$ (see \cite{LT});
i.e., $\mathfrak U(c_0,Y) = \mathfrak W(c_0,Y)$ for every Banach
space $Y$; but weakly compact operators on $c_0$ must be compact,
hence  $\mathfrak U(c_0,Y) = \mathfrak K(c_0,Y)$ for every Banach
space $Y$. This means that every unconditionally summing $\psi$
must be compact, hence, by Theorem \ref{eleinf}, it can be lifted
to an operator $c_0 \to \ell_\infty$, which means that the middle
sequence splits.\\

\emph{The ideal $\mathfrak X$ of separable range operators is not
balanced}. It is well-known the existence of nontrivial exact
sequences (see \cite{JL})
$$
\begin{CD}
0@>>>c_0@>>>X@>>>\ell_2(\Gamma)@>>>0.
\end{CD}$$ Consider a commutative diagram$$
\begin{CD}
0@>>> \ker q@>>>\ell_1(\Gamma)@>q>>\ell_2(\Gamma)@>>>0\\
 &&@V{\varphi}VV @VVV \Vert\\
0@>>>c_0 @>>>X@>>> c_0@>>>0\\
 &&\Vert && @VVV @VV{\psi}V \\
 0@>>> c_0 @>j>> \ell_{\infty} @>Q>> \ell_{\infty}/c_0 @>>>
 0.\end{CD}$$
Every $\varphi$ with range $c_0$ has separable range. On the other
hand, if some $\psi$ has separable range $S$ one gets a
commutative diagram
$$
\begin{CD}
0@>>> c_0@>>>\ell_\infty @>>>\ell_\infty/c_0 @>>>0\\
&&@| @AAA @AAA\\
0@>>> c_0@>>>W @>>> S @>>>0\\
&&@| @AAA @AA{\psi}A\\
0@>>>c_0@>>>X@>>> \ell_2(\Gamma)@>>>0.
 \end{CD}$$
But Sobczyk's theorem \cite{ccy} yields that the middle sequence
splits; hence the lower sequence must also
split.\\

\emph{The ideal $\mathfrak R$ of Rosenthal operators is not
balanced}. Recall that an operator $t$ is said to be a Rosenthal
operator if it maps bounded sequences into sequences admitting
weakly Cauchy subsequences. Consider an exact sequence $0\to \ker
q \to \ell_1 \stackrel{q}\to c_0 \to 0$ and apply the
Bourgain-Pisier construction \cite[Thm. 2.1]{BP} to $\ker q$ to
get an exact sequence $0\to \ker q \to \mathcal L_\infty(\ker q)
\to S \to 0$ in which $S$ has the Schur property. Combine all this
in a commutative push-out diagram$$
\begin{CD}
&&0&=&0\\
&&@VVV @VVV\\
  0@>>> \ker q @>>> \ell_1 @>>> c_0 @>>> 0\\
 & &@VVV @VVV \|\\
 0 @>>> \mathcal L_\infty(\ker q)@>>> PO @>>> c_0 @>>>0\\
 & &@VVV @VVV\\
 & &S &=&S\\
 &&@VVV @VVV\\
 &&0&=&0\end{CD}
$$ By a standard 3-space argument (see \cite[Thm. 6.1.a]{CG}) the space
$PO$ has the Schur property, hence the sequence  $0 \to \mathcal
L_\infty(\ker q) \to PO \to c_0\to 0$ cannot split. Consider now a
commutative diagram
$$
\begin{CD}
0@>>> \ker q@>>>\ell_1@>q>>c_0@>>>0\\
 &&@V{\varphi}VV @VVV \Vert\\
0@>>>\mathcal L_\infty(\ker q)@>>>PO@>>> c_0@>>>0\\
 &&\Vert && @VVV @VV{\psi}V \\
 0@>>> \mathcal L_\infty(\ker q) @>j>> \ell_{\infty} @>Q>> \ell_{\infty}/\mathcal L_\infty(\ker q) @>>>
 0.\end{CD}$$
On one hand, $\mathfrak L(c_0,Y)=\mathfrak R(c_0, Y)$ for every
Banach space $Y$. Hence $\psi$ must be a Rosenthal operator. On
the other hand, assume that some $\varphi$ is a Rosenthal
operator. Since weakly Cauchy sequences in Schur spaces are weakly
convergent, and weakly convergent sequences are convergent, the
operator $\varphi$ must be compact. Since the ideal $\mathfrak K$
is balanced, there must be some compact $\psi$ in the diagram
above. But since $\ell_\infty/\mathcal L_\infty(K)$ is an
$\mathcal L_\infty$-space \cite[Prop. 5.2.(c)]{LR}, it has the
BAP, and thus Proposition \ref{bap} yields that $\psi$ can be
lifted to an operator $\mathfrak L(c_0, \ell_\infty)$, which in
turn means that the middle sequence must split.
\end{proof}

\section{Appendix. The homological language}\label{homo}

Several of our results admit a homological formulation, which can
be clearer for those familiar with the language and methods of
homological algebra. We make now a brief exposition.

\subsection{Exact functors} Let $\mathcal A$ be a suitable space of functions/operators. Given
a fixed Banach space $X$ the functor $\mathcal A(X, {-})$ assigns
to a Banach space $Y$ the space $\mathcal A(X, Y)$ and to each
linear continuous operator $\tau: Y \to Y_1$ the ``composition-by
-the-left" map, namely $\tau_\circ: \mathcal A(X, Y) \To \mathcal
A(X, Y_1)$ defined by $\tau_\circ(f) = \tau \circ f$. The fact
that the direction of arrows is preserved is usually remarked by
saying that the functor is covariant. Dually, for fixed $Y$ one
has the contravariant functor $\mathcal A({-}, Y)$ assigning to
each Banach space $X$ the space $\mathcal A(X, Y)$ and to each
linear continuous operator $\tau: X \to X_1$ the ``composition-by
-the-right" map, namely $\tau^\circ: \mathcal A(X_1, Y) \To
\mathcal A(X, Y)$ defined by $\tau^\circ(f) = f \circ \tau$.\\

A  functor $\texttt{F}$ is said to be left-exact (resp.
right-exact) when given an exact sequence $0\rightarrow Y
\rightarrow X \rightarrow Z \rightarrow 0$ the induced sequence
$0\rightarrow \texttt{F}(Y) \rightarrow \texttt{F}(X) \rightarrow
\texttt{F}(Z)$ (resp. $\texttt{F}(Y) \rightarrow \texttt{F}(X)
\rightarrow \texttt{F}(Z) \rightarrow 0$) is exact. For instance,
it is well known  (see \cite{DF,Ka}) that the
functor $V\check{\otimes }_{\varepsilon }{-} $ (resp.
$V\hat{\otimes }_{\pi }{-} $) is left-exact (resp. right-exact);
or that given an injective and surjective operator ideal (see
\cite{Pi}) $\mathcal A$ the functor $\mathcal A(X, {-})$ is
left-exact and the functor $\mathcal A({-}, Y)$ is right-exact.

\adef We say that a functor $\texttt F$ is \emph{semi-exact} if it
transforms locally splitting sequences into exact sequences. We
will say it is \emph{locally exact} if it transforms locally
splitting sequences into locally splitting sequences.\zdef

Theorem \ref{locallysplit} and Proposition \ref{approximation}
immediately yield.

\begin{proposition} Given a Banach space $V$ one has:
\begin{itemize}
\item The functors $\mathfrak K({-}, V)$ and $\mathfrak W({-}, V)$
are semi-exact.\item The functors $\texttt{A}(V, {-})$ and
$\texttt{A}({-}, V)$ are locally-exact.
\end{itemize}
\end{proposition}

We have shown (combine Proposition \ref{approximation}, Theorem \ref{ls} and Aron-Berner extension) that
\begin{corollary} The functors $(V\check{\otimes}_{\varepsilon }{-})$ and
$(V\hat{\otimes }_{\pi }{-})$ are semi-exact if and only if they
are locally exact.
\end{corollary}

\subsection{Derivation} In many concrete situations, the involved functors
$\mathcal A(Z, {-})$ (resp. $\mathcal A({-},Y)$) are not exact.
Homology theory was developed to study this case. The measure of
how far a functor $\texttt F$ is from being exact is given by the
so-called \emph{derived} functor of $\texttt F$. The derivation
process  can be performed in different forms: using projective
presentations (when they exist), injective presentations (when
they exist), pull-back and push-out constructions, ...
\cite{HS,CG}.

Such is the case when one considers the category of Banach spaces
and the functor(s) induced by the operator ideal $\mathfrak L$,
namely $\mathfrak L(X, {-})$ and $\mathfrak L({-}, X)$. In which
case the derived functor(s) are called $\Ext$. See \cite{CC} for a
detailed exposition. All this means that if one has a Banach space
$Z$ and fixes a ``projective presentation" of $Z$; namely, an
exact sequence $0 \to \ker q \to \ell_1(I) \stackrel{q}\to Z \to
0$, and then fixes another Banach space $Y$ then there exists an
exact sequence
$$
0\longrightarrow\mathfrak L(Z, Y) \longrightarrow \mathfrak
L(\ell_1(I), Y) \longrightarrow \mathfrak L(\ker q,
Y)\longrightarrow \Ext_{\mathfrak L}^p(Z,Y)\longrightarrow 0.$$
This  can be taken as the definition of the space $\Ext_{\mathfrak
L}^p(Z,Y)$, which can be called the first derived functor of
$\mathfrak L({-}, {-})$; or, to be more precise, the first
projectively derived functor of $\mathfrak L({-}, {-})$. If,
instead, one fixes an ``injective presentation" of $Y$; namely,
exact sequences $0 \to Y \to \ell_\infty(\Lambda) \to Q \to 0$,
then there also exists an exact sequence
$$
0\longrightarrow \mathfrak L(Z, Y ) \longrightarrow \mathfrak L(Z,
\ell_\infty(\Lambda)) \longrightarrow \mathfrak L(Z, Q)
\longrightarrow \Ext_{\mathfrak L}^i(Z,Y)\longrightarrow 0.$$ Here
$\Ext_{\mathfrak L}^i({-}, {-})$ can be called the first derived
(injectively derived, to be precise) functor of $\mathfrak L({-},
{-})$. Classical homology theory starts with the fact \cite{HS}
that injective and projective derivations coincide; i.e.:
$$\Ext_{\mathfrak L}^p(Z,Y) = \Ext_{\mathfrak L}^i(Z,Y).$$ In fact,
a non-written rule \cite{HS} asserts that when one is working
in a reasonable category in which injective and projective
presentations exist then the derived functors obtained via
injective presentations and via projective presentations
coincide.\\

What happens when the functor to be derived is not the one induced
by the ideal $\mathfrak L$, but is induced by other operator
ideals $\mathcal A$. Let us show that making relative homology
with respect to an operator ideal requires to distinguish between
injective and projective derivation. Indeed, one needs to ask the
operator ideal to be injective (see \cite{Pi}) in order to compute
the relative $\Ext$ via injective presentations; and it has to be
projective (see \cite{Pi}) in order to compute the relative $\Ext$
using projective presentations. Let us define these functors to
then show that they can be quite different.

\begin{proposition} A projective operator ideal $\mathfrak A$ is left-exact.
\end{proposition}

This yields that given a projective operator ideal $\mathfrak A$,
a Banach space $Z$ and an exact sequence $0 \to \ker q \to
\ell_1(\Gamma) \stackrel{q}\to Z \to 0$, and given another Banach
space $Y$ then there exists an exact sequence
$$
0 \longrightarrow \mathfrak A(Z, Y) \longrightarrow\mathfrak
A(\ell_1(\Gamma), Y) \longrightarrow \mathfrak A(\ker q, Y)
\longrightarrow \diamondsuit \longrightarrow 0.$$

Straightforward computations show that $\diamondsuit$ does not depend on
the projective presentation of $Z$, and this allows us to define
$$\Ext_{\mathfrak A}^p(Z,Y) = \diamondsuit.$$

\begin{proposition} An injective operator ideal $\mathfrak A$ is right-exact.
\end{proposition}

This yields that given an injective operator ideal $\mathfrak A$,
a Banach space $Y$ and an exact sequence $0 \to Y \to
\ell_\infty(\Lambda) \to Q \to 0$, and given another Banach space
$Z$ then there exists an exact sequence
$$
0 \longrightarrow \mathfrak A(Z, Q) \longrightarrow \mathfrak A(Z,
\ell_\infty(\Lambda)) \longrightarrow \mathfrak A(Z, Y)
\longrightarrow \diamondsuit \longrightarrow 0.$$Straightforward
computations show that $\diamondsuit$ does not depend on the
injective presentation of $Y$, which allows us to define
$$\Ext_{\mathfrak A}^i(Z,Y) = \diamondsuit.$$

We have thus arrived to the core of the reason behind the
definition of balanced ideal. One has.

\begin{proposition} An injective and surjective
operator ideal $\mathfrak A$ is balanced if projective and
injective derivation coincide; equivalently, if the functors
$\Ext^{p}_{\mathfrak A}(\cdot,\cdot)$ and $\Ext^{i}_{\mathfrak
A}(\cdot,\cdot) $ are naturally equivalent.
\end{proposition}

It is easy now to translate the results of the previous section to
this language. For instance, the fact that $\mathfrak K$ is
balanced means that $\Ext_{\mathfrak K}$ is well defined, while
the fact that $\mathfrak W$ is not balanced means that one has to
consider two functors $\Ext_{\mathfrak W}^i$ and $\Ext_{\mathfrak
W}^p$, which can be very different. For instance, the examples in
the previous section  mean:

\begin{enumerate}
\item $\Ext_{\mathfrak W}^i(\ell_2, \mathcal L_\infty(D))\neq 0 =
\Ext_{\mathfrak W}^p(\ell_2, \mathcal L_\infty(D))$.

\item $\Ext_{\mathfrak U}^i(c_0, C[0,1])= 0 \neq \Ext_{\mathfrak
U}^p(c_0, C[0,1])$.

\item $\Ext_{\mathfrak X}^i(\ell_2(\Gamma), c_0)= 0 \neq
\Ext_{\mathfrak U}^p(\ell_2(\Gamma), c_0)$.

\item $\Ext_{\mathfrak R}^i(c_0, \mathcal L_\infty(\ker q))\neq 0
= \Ext_{\mathfrak R}^p(c_0, \mathcal L_\infty(\ker q))$.

\end{enumerate}

\end{document}